\documentclass[a4paper,12pt]{amsart}
\usepackage{times} 
\usepackage{latexsym,amscd,amssymb,url}
\usepackage[latin1]{inputenc} 	
\usepackage[T1]{fontenc}         
\usepackage{ae}									
\usepackage[all,cmtip]{xy}  
\usepackage{graphicx}
\usepackage{MnSymbol} 
\usepackage{color}

\newtheorem{theorem}{Theorem}

\theoremstyle{plain}

\newtheorem{corollary}[theorem]{Corollary}

\newtheorem{lemma}[theorem]{Lemma}

\newtheorem{proposition}[theorem]{Proposition}
\newtheorem{remark}[theorem]{Remark}

\newcommand{\Z}{\mathbb Z}
\newcommand{\Q}{\mathbb Q}
\newcommand{\PP}{\mathbb P}

\newcommand{\C}{\mathbb C}
\newcommand{\N}{\mathbb N}

\newcommand{\CP}{\mathbb P}

\newcommand{\Pic}{\operatorname{Pic}}

\newcommand{\pr}{\operatorname{pr}}

\newcommand{\Wh}{\operatorname{Wh}}

\newcommand{\dashedlongrightarrow}{\xymatrix@1@=15pt{\ar@{-->}[r]&}}
\renewcommand{\longrightarrow}{\xymatrix@1@=15pt{\ar[r]&}}
\renewcommand{\mapsto}{\xymatrix@1@=15pt{\ar@{|->}[r]&}}
\renewcommand{\twoheadrightarrow}{\xymatrix@1@=15pt{\ar@{->>}[r]&}}
\newcommand{\hooklongrightarrow}{\xymatrix@1@=15pt{\ar@{^(->}[r]&}}
\newcommand{\congpf}{\xymatrix@1@=15pt{\ar[r]^-\sim&}}
\renewcommand{\cong}{\simeq}

\begin{document}
\title[Algebraic structures with unbounded Chern numbers]{Algebraic structures with unbounded Chern numbers} 

\author{Stefan Schreieder}
\address{Mathematical Institute of the University of Bonn, Endenicher Allee 60, D-53115 Bonn, Germany.} 
\email{schreied@math.uni-bonn.de}

\author{Luca Tasin} 
\address{Mathematical Institute of the University of Bonn, Endenicher Allee 60, D-53115 Bonn, Germany.} 
\email{tasin@math.uni-bonn.de}

\date{March 30, 2016; \copyright{\ Stefan Schreieder and Luca Tasin 2015}}


\begin{abstract} 
We determine all Chern numbers of smooth complex projective varieties of dimension $\geq 4$ which are determined up to finite ambiguity by the underlying smooth manifold. 
We also give an upper bound on the dimension of the space of linear combinations of Chern numbers with that property and prove its optimality in dimension four.  
\end{abstract}

\maketitle

\section{Introduction}
To each $n$-dimensional complex manifold $X$ and for each partition $\mathfrak m$ of $n$, one can associate a Chern number $c_\mathfrak m(X)$. 
In 1954, Hirzebruch 
 asked which linear combinations of Chern and Hodge numbers are topological invariants of smooth algebraic varieties. 
Recently, this problem has been solved by Kotschick \cite{kotschick-PNAS,kotschick-advances} for what concerns the Chern numbers and by Kotschick and the first author \cite{kotschick-schreieder} in full generality.

Generalizing the Hirzebruch problem, Kotschick asks which  Chern numbers of smooth complex projective varieties are determined by the underlying smooth manifold up to finite ambiguity \cite[pp.\ 522]{kotschick-top}. 
Such a boundedness statement is known for $c_n$ and $c_1c_{n-1}$ in arbitrary dimension $n$, since these Chern numbers can be expressed in terms of Hodge numbers \cite{LW90} and so they are bounded by the Betti numbers.  
The first nontrivial instance of Kotschick's boundedness question concerns therefore the Chern number $c_1^3$ in dimension $3$.
In a recent preprint \cite{cascini-tasin}, Cascini and the second author show that in many cases this number is indeed bounded by the topology of the smooth projective threefold.

Conversely, there are no known examples of a smooth manifold such that the set of Chern numbers with respect to all possible complex algebraic structures is known to be unbounded. 
In this paper we produce such examples in dimensions $\geq 4$. 

\begin{theorem} \label{thm:unbounded}
In complex dimension $4$, the Chern numbers $c_4$, $c_1c_3$ and $c_2^2$ of a smooth complex projective variety are the only Chern numbers $c_{\mathfrak m}$ which are determined up to finite ambiguity by the underlying smooth manifold. 
In complex dimension $n\geq 5$, only $c_n$ and $c_1c_{n-1}$ are determined up to finite ambiguity by the underlying smooth manifold. 
\end{theorem}

The dimension four case of the above theorem might be surprising.
Indeed, it was observed by Kotschick that the Chern numbers of a minimal smooth projective fourfold of general type are bounded by the underlying smooth manifold, see Remark \ref{rem:3&4folds} below.
Based on an MMP approach, similar to the one given in \cite{cascini-tasin} for threefolds, one might expect that this boundedness statement holds more generally for all fourfolds of general type, which is the largest class in the Kodaira classification.
This compares to Theorem \ref{thm:unbounded} as the examples we are using there are of negative Kodaira dimension.

By Theorem \ref{thm:unbounded}, only very few Chern numbers of high dimensional smooth complex projective varieties are bounded by the underlying smooth manifold.
This changes considerably if we are asking for all linear combinations of Chern numbers with that property.  
Indeed, the space of such linear combinations contains the Euler characteristics $\chi^p=\chi(X,\Omega_X^p)$, as well as all Pontryagin numbers in even complex dimensions. 
In dimension four, the Euler characteristics $\chi^p$ and Pontryagin numbers span a space of codimension one in the space of all Chern numbers. Therefore, Theorem \ref{thm:unbounded} implies:

\begin{corollary} \label{cor:dim=4}
Any linear combination of Chern numbers which on smooth complex projective fourfolds is determined up to finite ambiguity by the underlying smooth manifold is a linear combination of the Euler characteristics $\chi^p$ and the Pontryagin numbers.
\end{corollary}

Using bordism theory, we provide in Corollary \ref{cor:upperbound} a nontrivial upper bound on the dimension of the space of linear combinations of Chern numbers which are determined up to finite ambiguity by the underlying smooth manifold.
Our upper bound is in general bigger than the known lower bound; determining all bounded linear combinations therefore remains open in all dimensions $n\geq 3$ other than $n=4$.  

It was known for some time that the boundedness question for Chern numbers behaves differently in the non-K\"ahler setting. 
Indeed, LeBrun showed \cite{lebrun} that there is a smooth $6$-manifold with infinitely many (non-K\"ahler) complex structures such that $c_1c_2$ is unbounded, which cannot happen for complex K\"ahler structures. 
In Corollary \ref{cor:cx:unbounded} we use products with LeBrun's examples and Theorem \ref{thm:unbounded} to conclude that in complex dimension $n\geq 4$, the topological Euler number $c_n$ is the only Chern number which on complex manifolds is bounded by the underlying smooth manifold.

Section \ref{sec:formulas} of this paper contains a systematic treatment of the Chern numbers of projective bundles. 
Theorem \ref{thm:unbounded} is based on these results and the existence of certain projective bundles over threefolds which admit infinitely many different algebraic structures.  
An important observation here is that the Chern numbers of the base do not matter too much.
To obtain unbounded Chern numbers for the projective bundles it is enough to have a three-dimensional base with unbounded first Chern class, its Chern numbers may well be independent of the complex structures chosen. 
This is in contrast to Kotschick's work \cite{kotschick-advances}, where bundles over surfaces with varying signatures are used, cf. Remark \ref{rem:dimB=2}. 

\section{Dolgachev surfaces} \label{sec:Dolgachev} 

We recall here some basic properties of Dolgachev surfaces. For a detailed treatment see  \cite{Dolgachev,FM94} and \cite[Sec.\ I.3]{FM88}.

Let $S\subseteq \CP^2\times \CP^1$ be a generic element of the linear series $|\mathcal O(3,1)|$. 
That is, $S$ is isomorphic to the blow-up of $\CP^2$ at the nine intersection points of two generic degree three curves and the second projection $\pi: S \longrightarrow \CP^1$ is an elliptic fibration with irreducible fibres. 
For each odd integer $q \geq 3$, the Dolgachev surface $S_q$ is realised applying logarithmic transformations of order $2$ and $q$ at two smooth fibres of $\pi$.  
The surface $S_q$ comes with an elliptic fibration $\pi_q:S_q \longrightarrow \CP^1$, which away from the two multiple fibers is isomorphic to the one of $S$.  
For a proof of the following proposition, see \cite[Sec.\ I.3]{FM88} and the references therein.

\begin{proposition} \label{prop:S_q}
The Dolgachev surface $S_q$ is a simply connected algebraic surface with
\begin{enumerate}
	\item $h^{2,0}(S_q)=0$ and $b_2(S_q)=10$,
	\item $c_1^2(S_q)=0$ and $c_2(S_q)=12$, 
	\item $c_1(S_q)=(q-2)G_q$, where $G_q\in H^2(S_q,\Z)$ is a nonzero primitive class, \label{item:c_1} 
	\item the intersection pairing on $H^2(S_q,\Z)$ is odd of type $(1,9)$. \label{item:(1,9)}
\end{enumerate} 
\end{proposition}

Proposition \ref{prop:S_q} has two important consequences that we will use in this paper.
Firstly, since $h^{1,0}(S_q)=h^{2,0}(S_q)=0$, it follows that the first Chern class is an isomorphism
\[
c_1 : \Pic (S_q) \stackrel{\sim}\longrightarrow H^2(S_q, \Z) .
\] 
Hence, every element of $H^2(S_q,\Z)$ can be represented by a holomorphic line bundle.

Secondly, let us denote the smooth manifold which underlies $S_q$ by $M_q$. 
By item (\ref{item:(1,9)}) in Proposition \ref{prop:S_q}, Wall's theorem \cite{Wall64} implies the existence of a smooth $h$-cobordism $W_q$ between $M_3$ and $M_q$.  

Although we will not need this here, let us mention that the homeomorphism type of $M_q$ does not depend on $q$ by Freedman's classification theorem of simply connected $4$-manifolds. 
However, generalizing a result of Donaldson, Friedman--Morgan showed \cite{FM88} that $M_q$ and $M_{q'}$ are never diffeomorphic for $q\neq q'$.

\section{Chern numbers of projective bundles} \label{sec:formulas}
In this section we systematically treat the Chern numbers of projective bundles.
Most of the results are taken from the first author's thesis \cite{bachelor}; we formulate and use them for holomorphic vector bundles over complex manifolds, but they hold more generally for arbitrary complex vector bundles over stably almost complex manifolds. 

Let $B$ be a complex manifold of dimension $n+1-k$ and let $E$ be a holomorphic vector bundle of rank $k$ on $B$.
The Segre class of $E$ is the inverse of its total Chern class; we denote it by
\[
\alpha:= (1+c_1(E)+\ldots +c_{k}(E))^{-1}\in H^\ast(B,\Z) .
\]
The degree $2k$-component of $\alpha$ is denoted by $\alpha_k\in H^{2k}(B,\Z)$.

For $\mathfrak a=(a_1 , \ldots ,a_p)\in \N^p$, we denote its weight by $|\mathfrak a|=\sum a_i$.
With this notation in mind, we put 
\begin{align} \label{def:f}
f(\mathfrak a):=\sum_{\mathfrak d \in \mathbb N^{p}} \left( \prod_{i=1}^{p}\binom{k-d_{i}}{k-a_{i}} c_{d_i}(E) \right) 		\alpha_{\left( |\mathfrak a|-|\mathfrak d| -(k-1) \right)}  ,
\end{align}
where $\mathfrak d=(d_1,\dots , d_p)$, and where we use the convention $\binom{a}{b}=0$, if $b<0$ or $a<b$.
The above definition yields a cohomology class in $H^{2(|\mathfrak a| -(k-1))} (B,\Q)$; it is motivated by the following result.  

\begin{proposition} \label{prop:c_m(PE)}
Let $\mathfrak m=(m_1 ,\ldots ,m_p )$ be a partition of $n=\dim(\CP(E))$.
Then the $\mathfrak m$-th Chern number of the projective bundle $\mathbb P (E)$ is given by
\[
c_{\mathfrak m}(\mathbb P (E))= \sum_{j_{1}, \ldots ,j_{p}} c_{j_{1}}(B)\cdot  \ldots \cdot  c_{j_{p}}(B) \cdot  f(m_{1}-j_{1}, \ldots ,m_{p}-j_{p})	,
\]
where the right hand side is identified with its evaluation on the fundamental class of $B$. 
\end{proposition}
\begin{proof}
Let $\pi: \mathbb P(E) \longrightarrow B$ be the projection morphism and $T_\pi$ be the tangent bundle along the fibres of $\pi$, that is, $T_\pi=\ker(\pi_*)$, where $\pi_*: T_{\PP(E)} \longrightarrow \pi^*T_X$.  
By the Whitney formula, the total Chern classes are related by
$$
c(\PP(E))= c(T_\pi) \cdot  \pi^*c(B).
$$
If $\mathcal O_E(-1)$ denotes the tautological bundle of $\PP(E)$, then we have the exact sequence
$$
0 \longrightarrow \mathcal O_E(-1) \longrightarrow \pi^\ast E \longrightarrow T_\pi \otimes \mathcal O_E(-1) \longrightarrow 0 .
$$
It follows that the total Chern classes of $T_\pi$ and $\pi^\ast E\otimes \mathcal O_E(1)$ coincide.
Hence, 
$$
c(T_\pi)= \sum_{i=0}^{k}\pi^\ast c_i(E)(1+y)^{k-i}, 
$$
where $y=c_1(\mathcal O_E(1))$. 
Setting $b_i:=\pi^* c_i(B)$ and $e_i:=\pi^* c_i(E)$, we can write
$$
c(\PP(E))=\left (\sum_{j \ge 0} b_j\right)\left(\sum_{i \ge 0} e_i(1+y)^{k-i}\right) ,
$$
and so
$$
c(\PP(E))= \sum_{i,j,l \ge 0} \binom{k-i}{l} e_ib_jy^l.
$$

The $\mathfrak m$-th Chern number is hence given by
$$
c_{\mathfrak m}(\PP(E))=\prod_{t=1}^p \sum_{i_t+j_t+l_t=m_t} \binom{k-i_t}{l_t} e_{i_t} b_{j_t}y^{l_t},
$$
where $i_t, j_t, l_t \ge 0$, and where we identify the right hand with its evaluation on the fundamental class of $\CP(E)$.  
Substituting $l_t=m_t-i_t-j_t$, we obtain
\begin{align*}
c_{\mathfrak m}(\PP(E)) 
&=\prod_{t=1}^p \sum_{i_t,j_t} \binom{k-i_t}{m_t-i_t-j_t} e_{i_t} b_{j_t}y^{m_t-i_t-j_t} \\
&=\prod_{t=1}^p \sum_{i_t,j_t} \binom{k-i_t}{k+j_t-m_t} e_{i_t} b_{j_t}y^{m_t-i_t-j_t} \\
&=\sum_{j_1, \ldots, j_p}   \sum_{i_1, \ldots, i_p}   \prod_{t=1}^p \binom{k-i_t}{k+j_t-m_t}e_{i_t}b_{j_t} y^{m_t-i_t-j_t} \\
&=\sum_{j_1, \ldots, j_p}  \sum_{i_1, \ldots, i_p}  \left( \prod_{t=1}^p \binom{k-i_t}{k+j_t-m_t}e_{i_t}b_{j_t}\right)y^{\sum_{t=1}^p  (m_t-i_t-j_t)} \\
&= \sum_{j_1, \ldots, j_p} b_{j_1} \cdots b_{j_p} \sum_{i_1, \ldots, i_p}  \left( \prod_{t=1}^p \binom{k-i_t}{k+j_t-m_t}e_{i_t}\right)y^{\sum_{t=1}^p ( m_t-i_t-j_t)} .
\end{align*}

For any $0 \le m \le n$ and any $\omega \in H^{2(n-m)}(B, \Z)$, the product $\omega y^m$ coincides with the top-degree component of $\omega \alpha y^{k-1}$, see \cite[Lem.\ 2.2]{schreieder}.
This simplifies the above expression of the $\mathfrak m$-th Chern number of $\CP(E)$ to
$$
c_{\mathfrak m}(\PP(E))= \sum_{j_1, \ldots, j_p} b_{j_1} \cdots b_{j_p} \sum_{i_1, \ldots, i_p} \left( \prod_{t=1}^p \binom{k-i_t}{k+j_t-m_t}e_{i_t}\right)\alpha y^{k-1} ,
$$
where on the right hand side only the term in cohomological degree $2n$ is considered.
The statement follows since on any fibre of $\pi$ the class $y^{k-1}$ evaluates to $1$.
\end{proof}

Proposition \ref{prop:c_m(PE)} reduces the computation of Chern numbers of projective bundles to the computation of $f(\mathfrak a)$ defined in (\ref{def:f}).
It is easy to see that $f(\mathfrak a)$ is invariant under permutations of $(a_{1}, \ldots ,a_{p})$. 
In fact, if $\sigma$ is a bijection of $\{1,\ldots,p\}$, we have
\begin{align*}
f(a_1,\ldots,a_p)&=\sum_{\mathfrak d \in \mathbb N^{p}} \left( \prod_{i=1}^{p}\binom{k-d_{i}}{k-a_{i}} c_{d_i}(E) \right) 		\alpha_{\left( |\mathfrak a|-|\mathfrak d| -(k-1) \right)} \\
&=\sum_{\mathfrak d \in \mathbb N^{p}} \left( \prod_{i=1}^{p}\binom{k-d_{\sigma(i)}}{k-a_{\sigma(i)}} c_{d_{\sigma(i)}}(E) \right) 		\alpha_{\left( |\mathfrak a|-|\mathfrak d| -(k-1) \right)}=f(a_{\sigma(1)},\ldots,a_{\sigma(p)}),
\end{align*}
where $\mathfrak d=(d_1,\dots , d_p)$.
Moreover, $f(\mathfrak a)$ is possibly nonzero only for $k-1\leq|\mathfrak a|\leq n$ and $0\leq a_i\leq k$, and a simple argument shows $f(\mathfrak a)=0$ for $a_i=k$.  
For small values of $|\mathfrak a|$, we are able to compute $f(\mathfrak a)$ explicitly as follows.

\begin{lemma} \label{lem:f}
Denoting by  $e_i:=c_i(E)$ the $i$-th Chern class of $E$, we have the following:
\begin{enumerate} 
	\item $f(\mathfrak a)= \prod_{i=1}^{p}{\binom{k}{a_{i}}} $ ,  if $|\mathfrak a| = k-1$, \label{item:weight=k-1}
	\item $f(\mathfrak a)= 0 $  
	,  if $|\mathfrak a| = k$, \label{item:weight=k}
	\item $f(\mathfrak a)= \left(  \prod_{i=1}^{p}{\binom {k }{ a_{i}}} \right)  \cdot \left( \left( \sum_{s<t}a_sa_t\right) -k \right)  \cdot \left(\frac{1}{k^{2}}e_{1}^{2}-\frac{2}{k(k-1)} e_{2} \right) $ ,  if $|\mathfrak a|= k+1$. \label{item:weight=k+1}
\end{enumerate}
\end{lemma}

\begin{proof}
The first assertion is immediate from the definition.
The second assertion can either be checked by a computation, or, alternatively one can argue as follows.
For any line bundle $L$ on $B$, $\CP(E)$ and $\CP(E\otimes L)$ are isomorphic. 
For $|\mathfrak a| = k$ the expression $f(\mathfrak a)$ has cohomological degree two and so it is a multiple of $e_1$. 
Specializing the base manifold $B$ to an elliptic curve, Proposition \ref{prop:c_m(PE)} shows that  for any line bundle $L$ on $B$, $f(\mathfrak a)$ is invariant under replacing $E$ by $E\otimes L$. 
The claim follows because no nontrivial multiple of $e_1$ has this property. 

It remains to prove (3). 
Since $|\mathfrak a|=k+1$, we have
$$
f(\mathfrak a) = \sum_{|\mathfrak d|=0} \left( \prod_{i=1}^p \binom{k-d_{i}}{k-a_{i}} e_{d_i} \right) \alpha_2 +
\sum_{|\mathfrak d|=1} \left( \prod_{i=1}^p \binom{k-d_{i}}{k-a_{i}} e_{d_i} \right) \alpha_1 + \sum_{|\mathfrak d|=2} \left( \prod_{i=1}^p \binom{k-d_{i}}{k-a_{i}} e_{d_i} \right) \alpha_0,
$$
which gives
\begin{align*}
f(\mathfrak a) &= \left( \prod_{i=1}^p \binom{k}{a_i} \right)  \left( \alpha_2 + \sum_{s=1}^p \frac{a_s}{k} e_1 \alpha_1 + \sum_{s=1}^{p} \frac{a_s(a_s-1)}{k(k-1)} e_2\alpha_0 + \sum_{ s < t } \frac{a_sa_t}{k^2}e_1^2\alpha_0 \right).
\end{align*}
Noting that 
$$
\alpha_1=-e_1 \quad \mbox{and} \quad \alpha_2=e_1^2-e_2,
$$
we can write
\begin{align*}
f(\mathfrak a) &= \left(  \prod_{i=1}^p \binom{k}{a_i} \right)  \left( \left(\sum_{ s < t}a_sa_t - \sum_{s=1}^p a_s k +k^2 \right)\frac{e_1^2}{k^2} + \left( \sum_{s=1}^p a_s(a_s-1) - k(k-1)  \right)\frac{e_2}{k(k-1)}\right) .
\end{align*}
The result follows now easily from $\sum_{s=1}^p a_s=k+1$ and $\sum_{s=1}^p a_s^2=(k+1)^2-2\sum_{s<t} a_sa_t$.
\end{proof}

In the construction of our examples, we will need the following easy estimate, which proves positivity of the constant appearing in $f(\mathfrak a)$ for $|\mathfrak a|=k+1$.

\begin{lemma} \label{lem:ineq}
Let $k\geq 2$ be an integer.
For any partition $\mathfrak a =(a_1,\ldots ,a_p)$ of $k+1$ with $0\leq a_i \leq k$ for all $i$, the expression
\begin{align} \label{eq:ineq}
\left( \prod_{i=1}^{p}{\binom{k}{a_{i}}} \right)  \cdot \left( \sum_{s<t} a_sa_t-k \right) 
\end{align}
from Lemma \ref{lem:f} is nonnegative; 
it is positive if additionally $a_i<k$ for all $i$.
\end{lemma}

\begin{proof}
The product $\prod_{i=1}^{p}{\binom{k}{a_{i}}}$ is positive since $0\leq a_i \leq k$ for all $i$.
It thus suffices to consider
\begin{align} \label{eq:aiaj-k}
\sum_{s<t} a_s a_t -k .
\end{align} 
Here we may ignore all $a_s$ that are zero. 
After reordering, we may therefore assume $1\leq a_1\leq a_2\leq \ldots \leq a_p\leq k$. 

If $p=2$, then 
\[
a_1\cdot a_2-k=a_1(k+1-a_1)-k 
\]
is a negatively curved quadratic equation in $a_1$ with zeros at $a_1=k$ and $a_1=1$ and so the assertion follows because $a_1=1$ implies $a_2=k$. 

If $p\geq 3$, then
\[
\sum_{s< t} a_sa_t \geq \sum_{s=2}^{p}a_1a_s +a_{p}a_{p-1}\geq \sum_{s=2}^{p}a_s +a_{1}=k+1 >k .
\]
Thus, (\ref{eq:aiaj-k}) is positive, which finishes the prove of the lemma.
\end{proof}

\section{Proof of Theorem \ref{thm:unbounded}} \label{sec:thm1}

In the notation of Section \ref{sec:Dolgachev}, for any odd integer $q \geq 3$  we have a smooth h-cobordism $W_q$ between $M_3$ and $M_q$ which induces an isomorphism $H^2(S_3, \Z) \cong H^2(S_q,\Z)$. 
Using this isomorphism, we fix a class
\[
\omega\in H^2(S_3,\Z)  \cong H^2(S_q, \Z)
\] 
of positive square. 
Since the intersection pairing on $S_3$ has type $(1,9)$, it follows that the orthogonal complement of $\omega$ is negative definite.
Hence, $G_q^2=0$ implies 
\[
\omega\cdot G_q \ne 0
\]
for all $q$.  
Via the first Chern class, each $S_q$ carries a unique holomorphic line bundle $L_q$ with $c_1(L_q) = \omega$.  

Let $C$ be a smooth curve of genus $g\geq 0$ and consider the threefold 
\[
Y_q:=S_q\times C .
\]
This threefold carries the holomorphic vector bundle
\begin{align} \label{eq:E_q}
E_q:= \left(\pr_1^\ast(L_q)\otimes \pr_2^\ast \mathcal O_C(1)\right) \oplus \mathcal O_{Y_q}^{\oplus r}
\end{align}
of rank $r+1$, where $\mathcal O_C(1)$ denotes some degree one line bundle on $C$.
The projectivization
\[
X_q:=\CP(E_q)
\]
is a smooth complex projective variety of dimension $n:=r+3$.

\begin{proposition} \label{prop:unbounded} 
If $n\geq 3$, then the oriented diffeomorphism class of the smooth manifold which underlies $X_q$ is independent of $q$.
If $n=4$, then the Chern numbers $c_1^4(X_q)$ and $c_1^2c_2(X_q)$ are unbounded in $q$.
If $n\geq 5$, then the $\mathfrak m$'s Chern number $c_{\mathfrak m}(X_q)$ is unbounded in $q$ for all partitions $\mathfrak m=(m_1,\ldots ,m_p)$ of $n$ with $1\leq m_i\leq n-2$ for all $i$.
\end{proposition}

\begin{proof}
We first prove the assertion concerning the diffeomorphism type of the manifold which underlies $X_q$; this part of the proof follows an argument used in \cite{kotschick-top} and \cite{kotschick-advances}.

Fix an odd integer $q \ge 3$ and consider the h-cobordism $W_q$.  
It follows from the exponential sequence for smooth functions that complex line bundles on $W_q$ are classified by $H^2(W_q,\Z)$.
Hence, we can find a complex line bundle $\mathbb L$ on $W_q$ with 
\[
c_1(\mathbb L)=\omega\in  H^2(S_3,\Z) \cong H^2(W_q,\Z) .
\]
Since the isomorphism $H^2(S_3, \Z) \cong H^2(S_q,\Z)$ is induced by $W_q$, it follows that the restriction of $\mathbb L$ to each of the boundary components of $W_q$ coincides with the complex line bundle which underlies the holomorphic line bundle $L_3$ resp.\ $L_q$ on $S_3$ resp.\ $S_q$.

Let us first consider the case $C\cong \CP^1$.
The product $W_q\times \CP^1$ is a simply connected h-cobordism between $M_3\times \CP^1$ and $M_q\times \CP^1$. 
It carries the complex vector bundle
\[
\mathbb E:=\left( \pr_1^\ast \mathbb L\otimes \pr_2^\ast \mathcal O_{\CP^1}(1)\right)\oplus \underline{\C}^{\oplus r} .
\]
The restrictions of this bundle to the boundary components of $W_q\times \CP^1$ coincide with the complex vector bundle which underlies the holomorphic vector bundle in (\ref{eq:E_q}).
Hence, the projectivization $\mathbb P(\mathbb E)$ is a simply connected h-cobordism between the simply connected oriented $2n$-manifolds which underly $X_3$ and $X_q$.
It thus follows from the h-cobordism theorem \cite{smale} that these smooth $2n$-manifolds are orientation-preserving diffeomorphic, as we claimed.

The above argument proves the first assertion in the proposition for $g=0$. 
For $g \ge 1$, one can use the s-cobordism theorem \cite{kervaire}. 
More precisely, since $\pi_1(M_q \times C)=\pi_1(C)$ and since the Whitehead group $\Wh(\pi_1(C))$ is trivial \cite[Thm.\ 1.11]{FJ90}, the s-cobordism theorem applies and we can conclude as before. 

In order to prove the second assertion, we use the computational tools given in Proposition \ref{prop:c_m(PE)} and Lemma \ref{lem:f} together with the positivity result in Lemma \ref{lem:ineq}.
Note that it suffices to compute $c_{\mathfrak m}(X_q)$ modulo all terms that do not depend on $q$.
For ease of notation, we identify cohomology classes on $S_q$ via pullback with classes on $Y_q$.
Using this notation, and fixing a point $c\in C$, we obtain
\begin{align*}
c_1(Y_q)&= c_1(S_q)+ (2-2g)\cdot[S_q\times c] , \\
c_2(Y_q)&= c_2(S_q) + (2-2g)\cdot c_1(S_q)\cdot [S_q\times c] , \\
c_3(Y_q)&= (2-2g)\cdot c_2(S_q) \cdot [S_q\times c] .
\end{align*}
In the above formulas, only $c_1(S_q)=(q-2)G_q$ depends on $q$.

In the notation of Proposition \ref{prop:c_m(PE)} and Lemma \ref{lem:f}, the rank of $E_q$ is denoted by $k=r+1$.
Recall that for any partition $\mathfrak a$ of $r+i$ the class $f(\mathfrak a)$ is a cohomology class in $H^{2i}(Y_q)$.
By Lemma \ref{lem:f}, this class is always independent of $q$, and it vanishes if additionally $i=1$.
For any partition $\mathfrak m=(m_1,\ldots ,m_p)$ of $n=r+3$ with $m_i\geq 1$ for all $i$, the $\mathfrak m$-th Chern number of $X_q$ is computed in Proposition \ref{prop:c_m(PE)}. 
Using Lemma \ref{lem:f}, we obtain 
\begin{align} \label{eq:prop:cm(PE)}
c_{\mathfrak m}(X_q) =
										 c_1(Y_q)\cdot \sum_{\mathfrak j}f(m_1-j_1,\ldots, m_p-j_p) + r(\mathfrak m) , 
\end{align} 
where $\mathfrak j=(j_1,\ldots ,j_p)$ runs through all partitions of $1$ by nonnegative integers, and where $r(\mathfrak m)$ is an integer which depends  on the partition $\mathfrak m$ of $n$ but not on $q$.
Explicitly,
\begin{align*}
r(\mathfrak m) :=& f(m_1,\ldots,m_p) + c_1(Y_q)c_2(Y_q) \cdot \sum_{\mathfrak h}f(m_1-h_1,\ldots, m_p-h_p) \\
&+ c_3(Y_q)\cdot \sum_{\mathfrak l}f(m_1-l_1,\ldots, m_p-l_p),
\end{align*}
where $\mathfrak h=(h_1,\ldots ,h_p)$ runs through all partitions of $3$ by nonnegative integers such that $h_i=2$ for one $i \in \{1,\ldots,p\}$, and $\mathfrak l=(l_1,\ldots ,l_p)$ runs through all partitions of $3$ by nonnegative integers such that $l_i=3$ for one $i \in \{1,\ldots,p\}$.
In this calculation we used that $c_1(Y_q)^3=0$ and that the formula for $c_{\mathfrak m}(X_q)$ has no nontrivial contribution by terms of the form $c_1(Y_q)^2\cdot f(\mathfrak a)$ or $c_2(Y_q)\cdot f(\mathfrak a)$, since $f(\mathfrak a)$ vanishes when $\mathfrak a$ has weight $|\mathfrak a|=k$, see Lemma \ref{lem:f}. 
In order to see that $r(\mathfrak m)$ does indeed not depend on $q$, it suffices to note that the terms $f(\mathfrak a)$, $c_1(Y_q)c_2(Y_q)$ and $c_3(Y_q)$ are all independent of $q$. 

By construction of $E_q$, we have $c_2(E_q)=0$ and
\[
c_1(E_q)=\omega + [S_q\times c] .
\]
This implies
\[
c_1(Y_q)\cdot c_1(E_q)^2=2(q-2)G_q\cdot \omega \cdot [S_q\times c] +(2-2g)\omega^2 \cdot [S_q\times c].
\]
This number is unbounded in $q$ since $G_q\cdot \omega$ is nonzero for all $q$ and the second summand does not depend on $q$.
It follows from Lemmas \ref{lem:f} and \ref{lem:ineq} that (\ref{eq:prop:cm(PE)}) is unbounded in $q$ as long as one of the partitions
\[
\mathfrak a:=(m_1-j_1,\ldots, m_p-j_p)
\]
that appears in (\ref{eq:prop:cm(PE)}) satisfies $m_i-j_i<k=n-2$.

If $n >4$, then this condition is equivalent to $m_i\leq n-2$ for all $i$. 

If $n=4$, then the above condition is only satisfied for $c_1^4$ and $c_1^2c_2$, as we want in the proposition.
\end{proof}

\begin{proof}[Proof of Theorem \ref{thm:unbounded}]
Recall that the Chern numbers $c_n$ and $c_1c_{n-1}$ are linear combinations of Hodge numbers \cite[Prop.\ 2.3]{LW90}, which on K\"ahler manifolds are bounded in terms of the Betti numbers of the underlying smooth manifold. 
Therefore, if $n\geq 5$, the theorem follows from Proposition \ref{prop:unbounded}.

In complex dimension $n=4$, the second Pontryagin number is given by
\begin{align} \label{eq:p2}
p_2=c_2^2 -2c_1c_3 + 2c_4.
\end{align}
This number depends only on the underlying oriented smooth $8$-manifold; changing the orientation changes $p_2$ by a sign.
Since $c_1c_3$ and $c_4$ are already known to be bounded by the underlying smooth manifold, the same conclusion holds for $c_2^2$. 
By Proposition \ref{prop:unbounded}, $c_1^4$ and $c_1^2c_2$ are unbounded, which finishes the proof of Theorem \ref{thm:unbounded}.
\end{proof}

\begin{remark} \label{rem:dimB=2} 
It easily follows from item (2) in Lemma \ref{lem:f} that the Chern numbers of a projective bundle over any surface remain bounded while changing the algebraic structure of the base. 
This explains why in our approach we had to use a base of dimension at least three.  
\end{remark}

\begin{remark} \label{rem:3&4folds}
The examples used in the proof of Theorem \ref{thm:unbounded} are ruled and so they have negative Kodaira dimension.
This compares to an observation of Kotschick which implies that in dimensions three and four the Chern numbers of a minimal projective manifold of general type are bounded by the underlying smooth manifold. 
Using the Miyaoka--Yau inequality, this boundedness statement was proven by Kotschick \cite[p.\ 522 and p.\ 525]{kotschick-top} under the stronger assumption of ample canonical class.
His argument applies because the inequality used holds more generally for arbitrary minimal projective manifolds of general type \cite{tsuji,zhang}.
\end{remark}

\begin{remark} \label{rem:friedman-morgan}
Koll\'ar \cite[Thm.\ 4.2.3]{kollar} proved that on a smooth manifold with $b_2=1$, the set of deformation equivalence classes of algebraic structures is finite, hence the Chern numbers are bounded.
Conversely, it was observed by Friedman--Morgan \cite{FM-bulletin} that the self-product of a Dolgachev surface yields an example of a smooth $8$-manifold where the set of deformation equivalence classes of algebraic structures is infinite because the order of divisibility of the canonical class can become arbitrarily large.
The Chern numbers of these examples are however bounded. 
\end{remark}

\section{Some applications}
The following corollary combines Theorem \ref{thm:unbounded} with LeBrun's examples \cite{lebrun}.

\begin{corollary} \label{cor:cx:unbounded}
In complex dimension $n\geq 4$, the topological Euler number $c_n$ is the only Chern number which on complex manifolds is bounded by the underlying smooth manifold.
\end{corollary}
\begin{proof}
The Chern number $c_n$ is clearly bounded by the underlying topological space.

Conversely, LeBrun \cite{lebrun} showed that there is a sequence $(Y_m)_{m\geq 1}$ of complex structures on the $6$-manifold $S^2\times M$, where $M$ denotes the $4$-manifold which underlies a complex K3 surface, such that $c_1c_2(Y_m)$ is unbounded, whereas $c_1^3(Y_m)$ and $c_3(Y_m)$ are both bounded.
It follows by induction on $n$ that $Y_m\times (\CP^1)^{n-3}$ has unbounded $c_1c_{n-1}$.
One also checks that $c_2^2(Y_m\times \CP^1)$ is unbounded.
This finishes the proof of Corollary \ref{cor:cx:unbounded} by Theorem \ref{thm:unbounded}. 
\end{proof}

It is not known whether on complex manifolds $c_1^3$ is bounded by the underlying smooth manifold.
As in the case of smooth complex projective varieties, $c_1^3$ is the only Chern number where unboundedness remains open. 
We emphasize however that Corollary \ref{cor:cx:unbounded} talks only about Chern numbers $c_{\mathfrak m}$ and not about their linear combinations.
The boundedness question for linear combinations of Chern numbers of complex manifolds remains open in general, but a partial result can be deduced from Section \ref{sec:upperbound} below.

The next two corollaries generalize an observation of Kotschick \cite[Rem.\ 20]{kotschick-advances}, asserting that the Chern number $c_1^n$ in dimension $n\geq 3$ does not lie in the span of the Euler characteristics $\chi^p$.

\begin{corollary} \label{cor:span1} 
A Chern number $c_{\mathfrak m}$ lies in the span of the Euler characteristics $\chi^p$ and the Pontryagin numbers if and only if  
\[
c_{\mathfrak m}\in \left\{c_1c_{n-1},c_n\right\} \ \ \text{or}\ \ c_{\mathfrak m}\in \left\{c_2^2,c_1c_3,c_4\right\} .
\] 
\end{corollary}

\begin{proof} 
The assertion is clear for $n\leq 2$, and it follows for $n=3$ because the space of the Euler characteristics $\chi^p$ is spanned by $c_1c_2$ and $c_3$, and there are no Pontryagin numbers.
If $n\geq 4$, then it follows immediately from Theorem \ref{thm:unbounded} and the fact that $c_1c_{n-1}$ and $c_n$ lie in the span of the Euler characteristics $\chi^p$, and $c_2^2$ lies in the span of the Euler characteristics and Pontryagin numbers in dimension four.
\end{proof}

\begin{corollary} \label{cor:span2}
The Chern numbers $c_1c_{n-1}$ and $c_n$ are the only Chern numbers that lie in the span of the $\chi^p$'s. 
No Chern number in even complex dimensions lies in the span of the Pontryagin numbers.
\end{corollary}

\begin{proof}
The fact that $c_1c_{n-1}$ and $c_n$ are the only Chern numbers that lie in the span of the $\chi^p$'s follows from Corollary \ref{cor:span1} and the observation that in dimension $n=4$, the span of the Euler characteristics $\chi^p$ has a basis given by $c_4$, $c_1c_3$ and $3c_2^2+4c_1^2c_2-c_1^4$, and so it does not contain $c_2^2$.

The assertion about the Pontryagin numbers in dimension $n=2$ follows from $p_1=c_1^2-2c_2$.
For $n\geq 4$, it suffices by Corollary \ref{cor:cx:unbounded} to show that $c_n$ is not a Pontryagin number.
This follows for example from \cite[Thm.\ 5]{kotschick-PNAS} and the fact that the signature is not a multiple of $c_n$.  
\end{proof}

\section{On the space of bounded linear combinations} \label{sec:upperbound}

In this section we give an upper bound on the dimension of the space of linear combinations of Chern numbers of smooth complex projective varieties that are bounded by the underlying smooth manifold.
For this purpose we determine the complex cobordism classes of the manifolds $X_q$ constructed in Section \ref{sec:thm1} in terms of suitable generators of $\Omega_\ast^U\otimes \Q$. 
This approach is based on the fact that in complex dimension $n$, the Chern numbers are complex cobordism invariants which form a basis of the dual space of $\Omega_n^U\otimes \Q$, see \cite[p.\ 117]{stong}.

Consider the elements $\alpha_1:=\CP^1$, $\alpha_2:=\CP^2$ and
\[
\alpha_n:=\CP(\mathcal O_{A}(1)\oplus \mathcal O_A^{n-2}) ,
\]
where $A$ denotes an abelian surface and $\mathcal O_A(1)$ denotes some ample line bundle on $A$. 
It follows from Lemma 2.3 in \cite{schreieder} that the Milnor number $s_n(\alpha_n)$ is nonzero.
By the structure theorem of Milnor and Novikov \cite[p.\ 128]{stong}, $(\alpha_n)_{n\geq 1}$ is therefore a sequence of generators of the complex cobordism ring with rational coefficients.  
That is, 
\[
\Omega_\ast^U\otimes \Q \cong \Q[\alpha_1,\alpha_2,\ldots] .
\]
Let us consider the bundle $E_q$ on $Y_q$ of rank $n-2$ and the corresponding $n$-dimensional projective bundle $X_q:=\CP(E_q)$ from Section \ref{sec:thm1}.  

\begin{proposition} \label{prop:X_q} 
There is an unbounded function $g_n(q)$ in $q$ such that the following identity holds in $\Omega_n^U\otimes \Q$:
\[
X_q=g_n(q)\cdot \alpha_1\alpha_{n-1} + \epsilon ,
\] 
where $\epsilon \in \Omega_{n}^{U}\otimes \Q$ denotes a rational cobordism class which does not depend on $q$.
\end{proposition}

\begin{proof}
Let $\mathfrak m$ be a partition of $n$.
By (\ref{eq:prop:cm(PE)}) and since $c_1(Y_q)=c_1(S_q)+(2-2g) [S_q\times c]$, we have 
\begin{align} \label{eq:c_m(Xq)}
c_{\mathfrak m}(X_q)=\sum_{|\mathfrak j|=1} c_{1}(S_q)\cdot   f(m_{1}-j_{1}, \ldots ,m_{p}-j_{p}) + r'(\mathfrak m),
\end{align}
where
$$
r'(\mathfrak m):=\sum_{|\mathfrak j|=1} (2-2g) [S_q\times c]\cdot   f(m_{1}-j_{1}, \ldots ,m_{p}-j_{p})+ r(\mathfrak m)
$$ 
is an integer which does not depend on $q$; in both summations, $\mathfrak j=(j_1,\ldots ,j_p)$ runs through all partitions of $1$  by nonnegative integers.

We now aim to compare the Chern numbers of $X_q$ with those of $\alpha_1 \alpha_{n-1}$.
To this end, let us consider the product $B:=\CP^1\times A$ together with the vector bundle
$
\pr_2^\ast \mathcal O_{A}(1)\oplus \mathcal O_B^{n-3} 
$. 
The projectivization 
\[
\CP(\pr_2^\ast \mathcal O_{A}(1)\oplus \mathcal O_B^{n-3})
\] 
has class $\alpha_1\alpha_{n-1}$ in $\Omega_\ast^U$.
By Proposition \ref{prop:c_m(PE)} we find
\[
c_{\mathfrak m}(\CP(\pr_2^\ast \mathcal O_{A}(1)\oplus \mathcal O_B^{n-3}))=f(m_{1}, \ldots ,m_{p}) + \sum_{|\mathfrak j|=1} c_{1}(B)\cdot   f(m_{1}-j_{1}, \ldots ,m_{p}-j_{p}) ,
\]
because $c_i(A)=0$ for all $i\geq 1$.
In the above calculation, $f(m_{1}, \ldots ,m_{p})$ is a cohomology class of degree $6$ which is actually a pullback from the second factor of $B$ and hence vanishes.
Comparing the above result with (\ref{eq:c_m(Xq)}) therefore proves
$$
c_{\mathfrak m}(X_q)=g_n(q)\cdot c_{\mathfrak m}(\alpha_1 \alpha_{n-1})+r'(\mathfrak m) ,
$$
for some rational number $g_n(q)$ which depends on $q$.
Since $r'(\mathfrak m)$ does not depend on $q$, it follows from Proposition \ref{prop:unbounded} that $g_n(q)$ is unbounded in $q$.

Since the Chern numbers in dimension $n$ form a basis of the dual space of $\Omega_{n}^{U}\otimes \Q $, there is a cobordism class $\epsilon \in \Omega_{n}^{U}\otimes \Q$ with $c_{\mathfrak m}(\epsilon)=r'(\mathfrak m)$ for all partitions $\mathfrak m$ of $n$.
Since $r'(\mathfrak m)$ does not depend on $q$, the same holds true for $\epsilon$.
Using the duality between the Chern numbers and $\Omega_{n}^{U}\otimes \Q $ once again, we deduce the identity
$$
X_q=g_n(q)\cdot \alpha_1\alpha_{n-1}+\epsilon
$$
in $\Omega_{n}^{U}\otimes \Q $.
This finishes the proof of the proposition, since $g_n(q)$ is unbounded in $q$. 
\end{proof}

Let us now consider the graded ideal
\[
\mathcal I^\ast:=\left\langle \alpha_1\alpha_k \mid k\geq 3\right\rangle
\]
in $\Omega_\ast^U\otimes \Q$.
By Proposition \ref{prop:X_q}, any linear combination of Chern numbers in dimension $n$ which on smooth complex projective varieties is bounded by the underlying smooth manifold vanishes on $\mathcal I^n$ and hence descends to the quotient $(\Omega_n^U\otimes \Q)/\mathcal I^n$.
Denoting by $p(n)$ the number of partitions of $n$ by positive natural numbers, we get the following. 

\begin{corollary} \label{cor:upperbound}
In dimension $n\geq 4$, the space of linear combinations of Chern numbers which on smooth complex projective varieties are bounded by the underlying smooth manifold has dimension at most 
\[
\dim(\Omega_n^U\otimes \Q)-\dim(\mathcal I^n)=p(n)-p(n-1)+\left\lfloor \frac{n+1}{2}\right\rfloor .
\]
\end{corollary}
\begin{proof}
We need to show that 
$$
\dim(\mathcal I^n)=p(n-1)-\left\lfloor \frac{n+1}{2}\right\rfloor.
$$
Clearly 
$$
\dim \left\langle \alpha_1\alpha_k \mid k\geq 1\right\rangle_{n} = p(n-1),
$$
and we have to subtract the number of partitions of $n-1$ by $1$ and $2$, which is 
$
\left\lfloor \frac{n+1}{2}\right\rfloor.
$   
This concludes the corollary.
\end{proof}

Finally, let us compare the upper bound from Corollary \ref{cor:upperbound} with the lower bound which is given by all Euler characteristics $\chi^p$ and all Pontryagin numbers in even complex dimension.
For this purpose, consider the ideal
\[
\mathcal J^\ast :=\left\langle \alpha_{2k+1}\mid k\geq 1\right\rangle + \left\langle \alpha_1\alpha_{2k} \mid k\geq 2\right\rangle
\] 
in $\Omega_\ast^U \otimes \Q$ which is generated by all $\alpha_{2k+1}$ with $k\geq 1$ and all $\alpha_1\alpha_{2k}$ where $k\geq 2$.
It is easily seen that the Euler characteristics $\chi^p$ as well as the Pontryagin numbers vanish on $\mathcal J^\ast$.
By \cite[Cor.\ 4]{kotschick-schreieder} the signature is the only linear combination of Pontryagin numbers which is contained in the span of the Euler characteristics $\chi^p$.
A simple dimension count therefore shows that the Euler characteristics and Pontryagin numbers in dimension $n$ form the dual space of 
\[
(\Omega_n^U \otimes \Q) \slash \mathcal J^n .
\]
We note that the inclusion $\mathcal I^n \subseteq \mathcal J^n$ is proper for all $n\geq 3$ with the exception of $n=4$, where equality holds.

\section*{Acknowledgments}
The first author is very grateful to his former advisor D.\ Kotschick for guidance during his bachelor thesis \cite{bachelor}, where most of the results of Section \ref{sec:formulas} were found. 
Both authors thank J.\ Koll\'ar for a stimulating question, P.\ Cascini, D.\ Huybrechts and B.\ Totaro for comments, and D.\ Kotschick and the anonymous referee for suggestions which improved the presentation of the results. 
During the preparation of the paper, the first author was member of the BIGS and supported by an IMPRS scholarship of the Max Planck Society; the second author is supported by the DFG Emmy Noether-Nachwuchsgruppe ``Gute Strukturen in der h\"oherdimensionalen birationalen Geometrie''. 
Both authors are member of the SFB/TR 45. 

\end{document}